\documentclass[10pt,reqno,a4paper,oneside]{amsart}%
\usepackage{amsfonts}
\usepackage{amssymb}
\usepackage{amsmath}
\usepackage{amsthm}
\usepackage{graphicx}%
\providecommand{\U}[1]{\protect\rule{.1in}{.1in}}
\newtheorem{theorem}{Theorem}[section]
\newtheorem{corollary}[theorem]{Corollary}
\newtheorem{lemma}[theorem]{Lemma}

\newtheorem{definition}{Definition}[section]
\newtheorem{remark}{Remark}[section]

\theoremstyle{definition}
\theoremstyle{remark}
\numberwithin{equation}{section}

\ifx\pdfoutput\relax\let\pdfoutput=\undefined\fi
\newcount\msipdfoutput
\ifx\pdfoutput\undefined\else
\ifcase\pdfoutput\else
\ifx\paperwidth\undefined\else
\ifdim\paperheight=0pt\relax\else\pdfpageheight\paperheight\fi
\ifdim\paperwidth=0pt\relax\else\pdfpagewidth\paperwidth\fi
\fi\fi\fi
\begin{document}
\pagestyle{myheadings}

\begin{center}
{\huge \textbf{Inverse eigenproblems and approximation problems for the
generalized reflexive and antireflexive matrices with respect to a pair of
generalized reflection matrices}}\footnote{This research was supported by the
Natural Science Foundation of China (11601328), the research fund of Shanghai Lixin University of Accounting and Finance (AW-22-2201-00118).}

\bigskip

{\large \textbf{Haixia Chang }}

$^{a}$School of Statistics and Mathematics, Shanghai Lixin University of Accounting and Finance,\newline
Shanghai 201209, P.R. China\\


\bigskip
\end{center}

\begin{quotation}
\textbf{Abstract} A matrix $P$ is said to be a nontrivial generalized
reflection matrix over the real quaternion algebra $\mathbb{H}$ if $P^{\ast
}=P\neq I$ and $P^{2}=I$ where $\ast$ means conjugate and transpose. We say
that $A\in\mathbb{H}^{n\times n}$ is generalized reflexive (or generalized
antireflexive) with respect to the matrix pair $(P,Q)$ if $A=PAQ$ $($or
$A=-PAQ)$ where $P$ and $Q$ are two nontrivial generalized reflection matrices
of demension $n$. Let ${\large \varphi}$ be one of the following subsets of
$\mathbb{H}^{n\times n}$ : (i) generalized reflexive matrix; (ii)reflexive
matrix; (iii) generalized antireflexive matrix; (iiii) antireflexive matrix.
Let $Z\in\mathbb{H}^{n\times m}$ with rank$\left(  Z\right)  =m$ and
$\Lambda=$ diag$\left(  \lambda_{1},...,\lambda_{m}\right)  .$ The inverse
eigenproblem is to find a\ matrix $A$ such that the set ${\large \varphi
}\left(  Z,\Lambda\right)  =\left\{  A\in{\large \varphi}\text{ }|\text{
}AZ=Z\Lambda\right\}  $ nonempty and find the general expression of
$A.$\newline In this paper, we investigate the inverse eigenproblem
${\large \varphi}\left(  Z,\Lambda\right)  $. Moreover, the approximation
problem: $\underset{A\in{\large \varphi}}{\min\left\Vert A-E\right\Vert _{F}}$
is studied, where $E$ is a given matrix over $\mathbb{H}$\ and $\parallel
\cdot\parallel_{F}$ is the Frobenius norm. \newline\textbf{Keywords} inner
inverse of a matrix; reflexive inverse of a matrix; generalized reflexive
matrix, generalized antireflexive matrix, Frobenius norm, approximation
problem\newline\textbf{2000 AMS Subject Classifications\ }{\small 15A24,
15A33, 15A57, 15A09}\newline
\end{quotation}

\section{Introduction}

Throughout $\mathbb{H}^{m\times n}$ denotes the set of all $m\times n$
matrices over the real quaternion algebra
\[
\mathbb{H=\{}a_{0}+a_{1}i+a_{2}j+a_{3}k\text{ }|\text{ }i^{2}=j^{2}%
=k^{2}=ijk=-1\text{ and }a_{0},a_{1},a_{2},a_{3}\ \text{are real
numbers}\mathbb{\}};
\]
the symbols $I,$ $A^{\ast},$ $r(A),$ $A^{\left(  1\right)  }$ $A^{+},$
$\parallel A\parallel_{F},$ stand for the identity matrix with the appropriate
size, the conjugate transpose, the rank, the inner inverse, the reflective
inverse of a matrix $A\in\mathbb{H}^{m\times n}$ and Frobenius norm of
$A\in\mathbb{H}^{m\times n}$, respectively; define inner product,
$\left\langle A,B\right\rangle =trace\left(  B^{\ast}A\right)  $ for all $A,$
$B\in\mathbb{H}^{m\times n},$ then $\mathbb{H}^{m\times n}$ is a Hilbert inner
product space and the norm that is generated by this inner product is
Frobenius norm. The two matrices $L_{A}$ and $R_{A}$ stand for $L_{A}%
=I-A^{+}A,$ $R_{A}=I-AA^{+}$ where $A^{+}$ is an any but fixed reflexive
inverse of $A.$ $V_{n}$ denotes the $n\times n$ backward identity matrix whose
elements along the southwest-northeast diagonal are ones and whose remaining
elements are zeros$.$

The real quaternion matrices play an important role in computer science,
quantum physicis, and so on (\emph{ }\cite{shuai}-\cite{sla}). The more
interests of quaternion matrices have been witnessed recently (e.g.,\emph{
}\cite{ZFZHEN}, \cite{DMV}). As $\mathbb{H}$ is noncommutative for its
mutiplication, i.e., $ab\neq ba,$ for $a,b\in\mathbb{H}$ in general, it bring
some obstacles in the study of quatenion matrices. Many authors has used the
method of embeddings of the quaternions into complex vector spaces, through
symplectic representations, such as D.R. Farenick and B. A.F. Pidkowich in
\cite{Farenick2003} and homotopy theory, such as Zhang in \cite{ZFZHEN}%
\emph{.}

Since left and right scalar multiplications are different, it force us to
consider $Ax=\lambda x,$ and $Ax=x\lambda$ separately. A quaternion $\lambda$
is said to a left (right) eigenvalue if $Ax=\lambda x$ $\left(  Ax=x\lambda
\right)  .$ In 1985, Wood proposed that every $n\times n$ quaternion matrix
has at least one left eigenvalue in $\mathbb{H}$ and Zhang in \cite{ZFZHEN}%
\emph{ }give its proof with the homotopy theory method. But the problem that
how many left eigenvalues a square quaternion matrix has has not been solved
yet. While the right eigenvalues of a square quaternion matrix has been well
studied, Zhang in \cite{ZFZHEN}\emph{ }explicitly refer to any $n\times n$
quaternion matrix $A$ has exactly $n$ right eigenvalues which are complex
numbers with nonnegative imaginary parts (standard eigenvalues of $A$).
Weighing the above cases, we only investigate the right inverse eigenproblems
of a square quaternion matrix in this paper.

Reflexive and antireflexive matrices with respect to a generalized reflection
matrix have been widely studied, e.g. \cite{hu2003}-\cite{Trench3}. But the
generalized reflexive and generalized antireflexive matrices with respect to
the generalized reflection matrix pair have not been widely studied yet.

In 1998, Chen in \cite{chen98} defined that $A\in\mathbb{H}^{m\times n}$ is
generalized reflexive (or generalized antireflexive) and give two special
subsets of the space $\mathbb{C}^{m\times n}$. The application of the
generalized reflexive (or generalized antireflexive) matrix is very wide in
engineering and science, such as the altitude estimation of a level network,
electric network and structural analysis of trusses and so on, see
\cite{chen98}.

\begin{definition}
A matrix $A\in\mathbb{H}^{n\times n}$ is generalized reflexive with respect to
the matrix pair $(P,Q)$ if $A=PAQ$ and $A\in\mathbb{H}^{n\times n}$ is
generalized antireflexive with respect to the matrix pair $(P,Q)$ if
$A=-PAQ,$where $P$ and $Q$ are two nontrivial generalized reflection matrices
of demension $n$, i.e., $P^{\ast}=P\neq I$, $P^{2}=I,$ and $Q^{\ast}=Q\neq I$,
$Q^{2}=I$.
\end{definition}

The generalized reflexive and the generalized antireflexive matrix are further
generalization of reflexive and antireflexive, respectively.\ From this
definition, it is clear that $A$ is a reflexive martix if generalized
reflection matrices $P=Q,$ i.e., $A=PAP$ and a centrosymmetric matrix if
$P=Q=V_{n},$ i.e., $A=V_{n}AV_{n}.$ Centrosymmetric and centroskew matrices
was investigated by many authors ,such as Andrew \cite{Andrew1},\cite{Andrew2}%
, Weaver \cite{Weaver}, Boullion and Atchison \cite{W.C.1}, Wang
\cite{wqw2005}, and the others.

Define two special subspaces of $\mathbb{H}^{n\times n}$
\begin{align*}
\mathbb{H}_{r}^{n\times n}\left(  P,Q\right)   &  =\left\{  A\text{ }|\text{
}A\in\mathbb{H}^{n\times n}\text{ and }A=PAQ\right\}  ,\\
\mathbb{H}_{a}^{n\times n}\left(  P,Q\right)   &  =\left\{  A\text{ }|\text{
}A\in\mathbb{H}^{n\times n}\text{ and }A=-PAQ\right\}  ,
\end{align*}
where $P$ and $Q$ are two generalized reflection matrices of dimension $m$ and
$n,$ respectively. If $A\in\mathbb{H}_{r}^{n\times n}\left(  P,Q\right)  $
$\left(  \in\mathbb{H}_{a}^{n\times n}\left(  P,Q\right)  \right)  ,$ then $A$
is a generalized reflexive (generalized antireflexive) matrix with respect to
nontrivial generalized reflection matrix pair $\left(  P,Q\right)  .$

The inverse eigenproblems of quaternion matrices: given matrices
$Z\in\mathbb{H}^{n\times m}$ and $\Lambda=$diag$\left(  \lambda_{1}%
,...,\lambda_{m}\right)  ,$ find $A\in\mathbb{H}^{n\times n}$ satisfying
$AZ=Z\Lambda,$ where $\lambda_{1},...,\lambda_{m}$ are complex numbers with
nonnegative imaginary parts. The inverse problems and inverse eigenproblems of
matrices has been wide topics, (e.g., \cite{Trench2}, \cite{bai}-\cite{XIE})
with the method involing Moore-Penrose inverse or singular value
decompositions of matrices related to $\left(  Z,\Lambda\right)  ,$ respectively.

In this paper, we investigate the inverse (right) eigenproblem of $\left(
Z,\Lambda\right)  ,$ find $A\in\mathbb{H}_{r}^{n\times n}\left(  P,Q\right)  $
$\left(  \in\mathbb{H}_{a}^{n\times n}\left(  P,Q\right)  \right)  $ and
present the general expression of $A.$ We make use of the structure properties
of $A\in\mathbb{H}_{r}^{n\times n}\left(  P,Q\right)  $ $\left(  \in
\mathbb{H}_{a}^{n\times n}\left(  P,Q\right)  \right)  $ to deal with the
inverse eigenproblem. Moreover, the approximation problem: $\underset
{A\in{\large \varphi}}{\min\left\Vert A-E\right\Vert _{F}},$ where $E$ is a
given matrix over $\mathbb{H}$ and\ ${\large \varphi}$ is one of
$\mathbb{H}_{r}^{n\times n}\left(  P,Q\right)  $ and $\mathbb{H}_{a}^{n\times
n}\left(  P,Q\right)  $. The explicit expression of $A$ is also presented.
Futhermore, it is pointed that some results in recent papers are special cases
of this paper.

\section{The solution of the inverse eigenproblems of $\left(  Z,\Lambda
\right)  $}

In this section, we first give a structure property of generalized reflexive
matrix and a generalized antireflexive matrix. Then we make use of the
structure property to derive the general expression of $A\in\mathbb{H}%
_{r}^{n\times n}\left(  P,Q\right)  $ $\left(  \in\mathbb{H}_{a}^{n\times
n}\left(  P,Q\right)  \right)  $ satisfying $AZ=Z\Lambda.$

By Proposition 2.1 in (\cite{Farenick2003}), we know that $R\in\mathbb{H}%
^{n\times n}$ has infinitely many nonreal right eigenvalues if $R$ has a
nonreal right eigenvalue. Because if $\left(  \lambda,\xi\right)  $ is a right
eigenvalue, then $\left(  w^{-1}\lambda w,\xi w\right)  $ is also a right
eigenpair of $R$. Then we introduce the orbit $\theta(\lambda),$ i.e., a class
from a partition of $\mathbb{H}$. For example, a similarity orbit
$\theta(\lambda)$ of $\lambda$:
\[
\theta(\lambda)=\left\{  w^{-1}\lambda w\text{ }|\text{ }w\in\mathbb{H}\text{,
}w\neq0\right\}
\]
and a conjugacy class $\theta(\lambda)$ of $\lambda$:%
\[
\theta(\lambda)=\left\{  \bar{w}\lambda w\text{ }|\text{ }w\in\mathbb{H}%
\text{, }\left\vert w\right\vert =1\right\}  .
\]

\begin{lemma}
$($See Theorem 3.3 in \cite{Farenick2003}$)$ If $R\in\mathbb{H}^{n\times n}$
is normal, then there are matrices $D,$ $U\in\mathbb{H}^{n\times n}$ such that
: \newline$\left(  1\right)  $ $U$ is unitary matrix, $D$ is a diagonal
matrix, and $U^{\ast}RU=D;$\newline$\left(  2\right)  $ each diagonal entry of
$D$ is a complex number contained in the closed upper halfplane $\mathbf{C}%
^{+}.$\newline$\left(  3\right)  $ $q\in\mathbb{H}$ is a right eigenvalue of
$R$ if and only if $q\in\theta(\lambda)$ for some diagonal element $\lambda$
of $D.$
\end{lemma}

\begin{lemma}
Suppose that $P\in\mathbb{H}^{n\times n}$ is a nontrivial generalized
reflection matrix; then there exists a unitary matrix $U\in\mathbb{H}^{n\times
n}$ such that
\begin{equation}
P=U\left[
\begin{array}
[c]{cc}%
I_{r_{1}} & 0\\
0 & -I_{n-r_{1}}%
\end{array}
\right]  U^{\ast}. \label{aa1}%
\end{equation}

\end{lemma}

\begin{proof}
A matrix $P\in\mathbb{H}^{n\times n}$ is a nontrivial generalized reflection
matrix, i.e., $P^{\ast}=P\neq I$, $P^{2}=I.$ Obviously, $P$ is a normal
matrix. According to the minimal polynomial of $P,$ we know 1, -1 are all
right engivalues (or standard eigenvalues ) of $P.$ By Lemma 2.1, there exists
a unitary matrix $U$ satisfying (\ref{aa1}).
\end{proof}

If a matrix $Q\in\mathbb{H}^{n\times n}$ is a nontrivial generalized
reflective matrix, by Lemma 2.2, we get
\begin{equation}
Q=V\left[
\begin{array}
[c]{cc}%
I_{r_{2}} & 0\\
0 & -I_{n-r_{2}}%
\end{array}
\right]  V^{\ast}, \label{ab1}%
\end{equation}
where $V$ is a unitary matrix.

We can use a staightforward method to obtain $U$ and $V$ mentioned above. In
fact, by $P^{2}=I,$ i.e., $\left(  I-P\right)  \left(  I+P\right)  =0=0$ or
$\left(  I+P\right)  \left(  I-P\right)  =0,$ we know the column vectors of
$I+P$ and $I-P$ are the eigenvector belonging to the right eigenvalues
$\lambda=1,$ $\lambda=-1$ of the matrix $P,$ respectively. We can only apply
the Gram-Schmidt process to the columns of $I+P$ and $I-P,$respectively. Let
\begin{align*}
&  I+P\text{ }\underrightarrow{Gram-Schmidt}\text{ }U_{1}\\
&  I-P\text{ }\underrightarrow{Gram-Schmidt}\text{ }U_{2}%
\end{align*}
and in (\ref{aa1})%
\begin{equation}
U=\left[
\begin{array}
[c]{cc}%
U_{1}, & U_{2}%
\end{array}
\right]  \label{A1}%
\end{equation}
where $U_{1}^{\ast}U_{1}=I_{r_{1}},$ $U_{2}^{\ast}U_{2}=I_{m-r_{1}},$
$U_{1}^{\ast}U_{2}=0.$ Similarily, we can get the factorizaton of $V$ in
(\ref{ab1})%
\begin{equation}
V=\left[
\begin{array}
[c]{cc}%
V_{1}, & V_{2}%
\end{array}
\right]  \label{A2}%
\end{equation}
where $V_{1}^{\ast}V_{1}=I_{r_{2}},$ $V_{2}^{\ast}V_{2}=I_{n-r_{2}},$
$V_{1}^{\ast}V_{2}=0.$

By (\ref{aa1})-(\ref{A2}), we obtain
\begin{align}
P  &  =\left[
\begin{array}
[c]{cc}%
U_{1}, & U_{2}%
\end{array}
\right]  \left[
\begin{array}
[c]{cc}%
I_{r_{1}} & 0\\
0 & -I_{n-r_{1}}%
\end{array}
\right]  \left[
\begin{array}
[c]{c}%
U_{1}^{\ast}\\
U_{2}^{\ast}%
\end{array}
\right]  ,\label{A3}\\
Q  &  =\left[
\begin{array}
[c]{cc}%
V_{1}, & V_{2}%
\end{array}
\right]  \left[
\begin{array}
[c]{cc}%
I_{r_{2}} & 0\\
0 & -I_{n-r_{2}}%
\end{array}
\right]  \left[
\begin{array}
[c]{c}%
V_{1}^{\ast}\\
V_{2}^{\ast}%
\end{array}
\right]  , \label{A4}%
\end{align}
where
\begin{equation}
U_{1}^{\ast}U_{1}=I_{r_{1}},U_{2}^{\ast}U_{2}=I_{m-r_{1}},U_{1}^{\ast}U_{2}=0
\label{A5}%
\end{equation}
and
\begin{equation}
V_{1}^{\ast}V_{1}=I_{r_{2}},V_{2}^{\ast}V_{2}=I_{n-r_{2}},V_{1}^{\ast}V_{2}=0.
\label{A6}%
\end{equation}

\begin{theorem}
A matrix $A\in\mathbb{H}^{n\times n}$ is generalized reflexive matrix with
respect to the nontrivial generalized reflection matrix pair $(P,Q)$ if and
only if $A$ can be expressed
\begin{equation}
A=\left[
\begin{array}
[c]{cc}%
U_{1}, & U_{2}%
\end{array}
\right]  \left[
\begin{array}
[c]{cc}%
A_{11} & 0\\
0 & A_{22}%
\end{array}
\right]  \left[
\begin{array}
[c]{c}%
V_{1}^{\ast}\\
V_{2}^{\ast}%
\end{array}
\right],  \label{aa2}%
\end{equation}
where
\begin{equation}
A_{11}=U_{1}^{\ast}AV_{1}\text{, }A_{22}=U_{2}^{\ast}AV_{2} \label{AS1}%
\end{equation}
and $U_{1},$ $U_{2},$ $V_{1},$ $V_{2}\ $are defined as (\ref{A1}) and
(\ref{A2}).
\end{theorem}

\begin{proof}
By the definition of $A\in\mathbb{H}_{r}^{n\times n}\left(  P,Q\right)  ,$
(\ref{A3}) and (\ref{A4}), we have
\begin{equation}
A=PAQ=\left[
\begin{array}
[c]{cc}%
U_{1}, & U_{2}%
\end{array}
\right]  \left[
\begin{array}
[c]{cc}%
I_{r_{1}} & 0\\
0 & -I_{n-r_{1}}%
\end{array}
\right]  \left[
\begin{array}
[c]{c}%
U_{1}^{\ast}\\
U_{2}^{\ast}%
\end{array}
\right]  A\left[
\begin{array}
[c]{cc}%
V_{1}, & V_{2}%
\end{array}
\right]  \left[
\begin{array}
[c]{cc}%
I_{r_{2}} & 0\\
0 & -I_{n-r_{2}}%
\end{array}
\right]  \left[
\begin{array}
[c]{c}%
V_{1}^{\ast}\\
V_{2}^{\ast}%
\end{array}
\right]  , \label{ab3}%
\end{equation}
i.e.,
\begin{equation}
\left[
\begin{array}
[c]{c}%
U_{1}^{\ast}\\
U_{2}^{\ast}%
\end{array}
\right]  A\left[
\begin{array}
[c]{cc}%
V_{1}, & V_{2}%
\end{array}
\right]  =\left[
\begin{array}
[c]{cc}%
I_{r_{1}} & 0\\
0 & -I_{n-r_{1}}%
\end{array}
\right]  \left[
\begin{array}
[c]{c}%
U_{1}^{\ast}\\
U_{2}^{\ast}%
\end{array}
\right]  A\left[
\begin{array}
[c]{cc}%
V_{1}, & V_{2}%
\end{array}
\right]  \left[
\begin{array}
[c]{cc}%
I_{r_{2}} & 0\\
0 & -I_{n-r_{2}}%
\end{array}
\right]  \label{ab2}%
\end{equation}
Put
\begin{equation}
\left[
\begin{array}
[c]{c}%
U_{1}^{\ast}\\
U_{2}^{\ast}%
\end{array}
\right]  A\left[
\begin{array}
[c]{cc}%
V_{1}, & V_{2}%
\end{array}
\right]  =\left[
\begin{array}
[c]{cc}%
A_{11} & A_{12}\\
A_{21} & A_{22}%
\end{array}
\right]  , \label{aa3}%
\end{equation}
where $A_{11}\in\mathbb{H}^{r_{1}\times r_{2}},$ $A_{12}\in\mathbb{H}%
^{r_{1}\times\left(  n-r_{2}\right)  },$ $A_{21}\in\mathbb{H}^{\left(
n-r_{1}\right)  \times r_{2}},$ $A_{22}\in\mathbb{H}^{\left(  n-r_{1}\right)
\times\left(  n-r_{2}\right)  }.$ Substituting (\ref{aa3}) into (\ref{ab2}),
we have
\[
\left[
\begin{array}
[c]{cc}%
A_{11} & A_{12}\\
A_{21} & A_{22}%
\end{array}
\right]  =\left[
\begin{array}
[c]{cc}%
I_{r_{1}} & 0\\
0 & -I_{n-r_{1}}%
\end{array}
\right]  \left[
\begin{array}
[c]{cc}%
A_{11} & A_{12}\\
A_{21} & A_{22}%
\end{array}
\right]  \left[
\begin{array}
[c]{cc}%
I_{r_{2}} & 0\\
0 & -I_{n-r_{2}}%
\end{array}
\right]  =\left[
\begin{array}
[c]{cc}%
A_{11} & -A_{12}\\
-A_{21} & A_{22}%
\end{array}
\right]  ,
\]
yielding $A_{12}=0$, $A_{21}=0.$ (\ref{ab3}) becomes (\ref{aa2}). By
(\ref{aa2}), (\ref{A5}) and (\ref{A6}), it is easy to get (\ref{AS1}).

Conversly, if (\ref{aa2}) holds it is easy to verify $A$ is generalized
reflexive matrix with respect to the nontrivial generalized reflection matrix
pair $(P,Q).$
\end{proof}

Similarly, we can obtain the factorization of generalized antireflexive matrix
with respect to the generalized reflection matrix pair $(P,Q).$ We have the
following Theorem.

\begin{theorem}
A matrix $A\in\mathbb{H}_{a}^{n\times n}\left(  P,Q\right)  $ if and only if
$A$ can be expressed as
\[
A=\left[
\begin{array}
[c]{cc}%
U_{1}, & U_{2}%
\end{array}
\right]  \left[
\begin{array}
[c]{cc}%
0 & A_{12}\\
A_{21} & 0
\end{array}
\right]  \left[
\begin{array}
[c]{c}%
V_{1}^{\ast}\\
V_{2}^{\ast}%
\end{array}
\right]  ,
\]
where
\[
A_{12}=U_{1}^{\ast}AV_{2}\text{, \ }A_{21}=U_{2}^{\ast}AV_{1},%
\]
and $U_{1},$ $U_{2},$ $V_{1},$ $V_{2}\ $are defined as (\ref{A1}) and
(\ref{A2}).
\end{theorem}

Now we discuss the solution of the inverse eigenproblems of $\left(
Z,\Lambda\right)  .$

\begin{lemma}
Given $B\in\mathbb{H}^{m\times t}$ $X\in\mathbb{H}^{s\times t}.$ Then the
inverse problem
\[
AX=B
\]
is consistent if and only if%
\[
BX^{+}X=B.
\]
In that case, the general form of $A\in\mathbb{H}^{n\times n}$ is
\[
A=BX^{+}+WR_{X}%
\]
where $W$ is an arbitrary matrix over $\mathbb{H}$ with appropriate size.
\end{lemma}

\begin{theorem}
Given $Z\in\mathbb{H}^{n\times m}$, $\Lambda=$diag$\left(  \lambda
_{1},...,\lambda_{m}\right)$. Let $U_{1},$ $U_{2},$ $V_{1},$ $V_{2} $ be defined as (\ref{A1}),
(\ref{A2}),
\begin{equation}
Z=\left[
\begin{array}
[c]{cc}%
V_{1}X_{1}, & V_{2}X_{2}%
\end{array}
\right]  =\left[
\begin{array}
[c]{cc}%
U_{1}Y_{1}, & U_{2}Y_{2}%
\end{array}
\right]  ,\text{ }\Lambda=\left[
\begin{array}
[c]{cc}%
\Phi & 0\\
0 & \Psi
\end{array}
\right]  ,\text{ } \label{AK1}%
\end{equation}
where $0<k<m$, $X_{1}\in\mathbb{H}^{r_{1}\times k},$ $X_{2}\in\mathbb{H}^{\left(
n-r_{1}\right)  \times\left(  m-k\right)  }$, $Y_{1}\in\mathbb{H}^{r_{2}\times
k},$ $Y_{2}\in\mathbb{H}^{\left(  n-r_{2}\right)  \times\left(  m-k\right)
},$ $\Phi=$diag$\left(  \lambda_{1},...,\lambda_{k}\right)  ,$ $\Psi
=$diag$\left(  \lambda_{k+1},...,\lambda_{m}\right)  .$ Then the set
\begin{equation}
{\large \varphi}_{1}\left(  Z,\Lambda\right)  =\left\{  A\in\mathbb{H}%
_{r}^{n\times n}\left(  P,Q\right)  \text{ }|\text{ }AZ=Z\Lambda\right\}
\label{ak1}%
\end{equation}
is nonempty if and only if
\begin{equation}
Y_{1}\Phi X_{1}X_{1}^{+}=Y_{1}\Phi,\text{ }Y_{2}\Psi X_{2}X_{2}^{+}=Y_{2}\Psi.
\label{AK2}%
\end{equation}
In that case, the general expression of $A$ is
\begin{equation}
A=\left[
\begin{array}
[c]{cc}%
U_{1}, & U_{2}%
\end{array}
\right]  \left[
\begin{array}
[c]{cc}%
Y_{1}\Phi X_{1}^{+}+W_{1}R_{X_{1}} & 0\\
0 & Y_{2}\Psi X_{2}^{+}+W_{2}R_{X_{2}}%
\end{array}
\right]  \left[
\begin{array}
[c]{c}%
V_{1}^{\ast}\\
V_{2}^{\ast}%
\end{array}
\right]  , \label{AK3}%
\end{equation}
where $W_{1},W_{2}$ are arbitrary matrices over $\mathbb{H}$ with
appropriate sizes.
\end{theorem}

\begin{proof}
Since $A\in\mathbb{H}_{r}^{n\times n}\left(  P,Q\right)  ,$ $A$ has the form
of (\ref{aa2}). By (\ref{A1}) and (\ref{A2}), substituting (\ref{aa2}) and
(\ref{AK1}) into $AZ=Z\Lambda.$ We get
\[
\left[
\begin{array}
[c]{cc}%
A_{11}X_{1} & 0\\
0 & A_{22}X_{2}%
\end{array}
\right]  =\left[
\begin{array}
[c]{cc}%
Y_{1}\Phi & 0\\
0 & Y_{2}\Psi
\end{array}
\right]  ,
\]
i.e.,
\[
A_{11}X_{1}=Y_{1}\Phi,\text{ }A_{22}X_{2}=Y_{2}\Psi.
\]
The set ${\large \varphi}\left(  Z,\Lambda\right)  $ is nonempty if and only
if the inverse problem $A_{11}X_{1}=Y_{1}$ and $A_{22}X_{2}=Y_{2}\Psi$ are
consistent. By Lemma 2.5, the invese problem $A_{11}X_{1}=Y_{1}$ and
$A_{22}X_{2}=Y_{2}\Psi$ are consistent if and only if (\ref{AK2}) holds and
their general solution are
\begin{align}
A_{11}  &  =Y_{1}\Phi X_{1}^{+}+W_{1}R_{X_{1}},\label{AK4}\\
A_{22}  &  =Y_{2}\Psi X_{2}^{+}+W_{2}R_{X_{2}},\nonumber
\end{align}
respectively. Then substituting (\ref{AK4}) into (\ref{aa2}), we get
(\ref{AK3}).

Conversly, if (\ref{AK2}) holds and $A$ has the expression of (\ref{aa2}), it
is easy to verify $A\in{\large \varphi}\left(  Z,\Lambda\right)  .$
\end{proof}

If $A\in\mathbb{H}_{r}^{n\times n}\left(  P,P\right)  ,$ i.e., $A$ is a
reflexive matrix with respect to a generalized reflection matrix $P.$ We have
the following Corollary.

\begin{corollary}
Given $Z\in\mathbb{H}^{n\times m}$, $\Lambda=$diag$\left(  \lambda
_{1},...,\lambda_{m}\right)  . $ Let $U_{1},$ $U_{2},$ $V_{1},$ $V_{2}\ $ be defined as (\ref{A1}) and
(\ref{A2}),
\[
Z=\left[
\begin{array}
[c]{cc}%
U_{1}X_{1}, & U_{2}X_{2}%
\end{array}
\right]  ,\text{ }\Lambda=\left[
\begin{array}
[c]{cc}%
\Phi & 0\\
0 & \Psi
\end{array}
\right]  ,\text{ }%
\]
where $0<k<m,$ $X_{1}\in\mathbb{H}^{r_{1}\times k},$ $X_{2}\in\mathbb{H}^{\left(
n-r_{1}\right)  \times\left(  m-k\right)  },$ $\Phi=$diag$\left(  \lambda
_{1},...,\lambda_{k}\right)  ,$ $\Psi=$diag$\left(  \lambda_{k+1}%
,...,\lambda_{m}\right)  .$ Then the set
\begin{equation}
{\large \varphi}_{1}^{^{\prime\prime}}\left(  Z,\Lambda\right)  =\left\{
A\in\mathbb{H}_{r}^{n\times n}\left(  P,P\right)  \text{ }|\text{ }%
AZ=Z\Lambda\right\}  \label{ak4}%
\end{equation}
is nonempty if and only if
\[
X_{1}\Phi X_{1}X_{1}^{+}=X_{1}\Phi,\text{ }X_{2}\Psi X_{2}X_{2}^{+}=X_{2}%
\Psi.
\]
In that case, the general expression of $A$ is
\[
A=\left[
\begin{array}
[c]{cc}%
U_{1}, & U_{2}%
\end{array}
\right]  \left[
\begin{array}
[c]{cc}%
X_{1}\Phi X_{1}^{+}+W_{1}R_{X_{1}} & 0\\
0 & X_{2}\Psi X_{2}^{+}+W_{2}R_{X_{2}}%
\end{array}
\right]  \left[
\begin{array}
[c]{c}%
U_{1}^{\ast}\\
U_{2}^{\ast}%
\end{array}
\right]  ,
\]
where $W_{1},W_{2}$ are
arbitrary matrices over $\mathbb{H}$ with appropriate sizes.
\end{corollary}

\begin{theorem}
Given $Z\in\mathbb{H}^{n\times m}$, $\Lambda=$diag$\left(  \lambda
_{1},...,\lambda_{m}\right)$. Let $U_{1},$ $U_{2},$ $V_{1},$ $V_{2} $ be defined as (\ref{A1}),
(\ref{A2}),
\[
Z=\left[
\begin{array}
[c]{cc}%
V_{2}X_{1}, & V_{1}X_{2}%
\end{array}
\right] =\left[
\begin{array}
[c]{cc}%
U_{1}Y_{1}, & U_{2}Y_{2}%
\end{array}
\right]  ,\text{ }\Lambda=\left[
\begin{array}
[c]{cc}%
\Phi & 0\\
0 & \Psi
\end{array}
\right]  ,\text{ }%
\]
where $0<l<m$, $X_{1}\in\mathbb{H}^{\left(  n-r_{2}\right)  \times l},$ $X_{2}%
\in\mathbb{H}^{r_{2}\times\left(  m-l\right)  },$ $Y_{1}\in\mathbb{H}%
^{r_{2}\times l},$ $Y_{2}\in\mathbb{H}^{\left(  n-r_{2}\right)  \times\left(
m-l\right)  },$ $\Phi=$diag$\left(  \lambda_{1},...,\lambda_{l}\right)  ,$
$\Psi=$diag$\left(  \lambda_{l+1},...,\lambda_{m}\right)  .$ Then the set
\begin{equation}
{\large \varphi}_{2}\left(  Z,\Lambda\right)  =\left\{  A\in\mathbb{H}%
_{a}^{n\times n}\left(  P,Q\right)  \text{ }|\text{ }AZ=Z\Lambda\right\}
\label{AG}%
\end{equation}
${\large \varphi}\left(  Z,\Lambda\right)  =\left\{  A\in\mathbb{H}%
_{a}^{n\times n}\left(  P,Q\right)  \text{ }|\text{ }AZ=Z\Lambda\right\}  $ is
nonempty if and only if
\[
Y_{1}\Phi X_{1}X_{1}^{+}=Y_{1}\Phi,\text{ }Y_{2}\Psi X_{2}X_{2}^{+}=Y_{2}%
\Psi.
\]
In that case, the general expression of $A$ is
\[
A=\left[
\begin{array}
[c]{cc}%
U_{1}, & U_{2}%
\end{array}
\right]  \left[
\begin{array}
[c]{cc}%
0 & Y_{1}\Phi X_{1}^{+}+W_{1}R_{X_{1}}\\
Y_{2}\Psi X_{2}^{+}+W_{2}R_{X_{2}} & 0
\end{array}
\right]  \left[
\begin{array}
[c]{c}%
V_{1}^{\ast}\\
V_{2}^{\ast}%
\end{array}
\right]  ,
\]
where $W_{1},W_{2}$ are arbitrary matrices over $\mathbb{H}$ with
appropriate sizes.
\end{theorem}

\begin{proof}
Similar to the proof of the Theorem 2.6.
\end{proof}

If $A\in\mathbb{H}_{a}^{n\times n}\left(  P,P\right)  ,$ i.e., $A$ is a
antireflexive matrix with respect to a generalized reflection matrix $P.$ We
have the following Corollary.

\begin{corollary}
Given $Z\in\mathbb{H}^{n\times m}$, $\Lambda=$diag$\left(  \lambda
_{1},...,\lambda_{m}\right)  .$ Let $U_{1},$ $U_{2}\ $ be defined as (\ref{A1}),
\[
Z=\left[
\begin{array}
[c]{cc}%
U_{1}X_{1}, & U_{2}X_{2}%
\end{array}
\right]  ,\text{ }\Lambda=\left[
\begin{array}
[c]{cc}%
\Phi & 0\\
0 & \Psi
\end{array}
\right]  ,\text{ }%
\]
where $0<l<m,$ $X_{1}\in\mathbb{H}^{\left(  n-r_{1}\right)  \times l},$ $X_{2}%
\in\mathbb{H}^{r_{1}\times\left(  m-l\right)  },$ $\Phi=$diag$\left(
\lambda_{1},...,\lambda_{l}\right)  ,$ $\Psi=$diag$\left(  \lambda
_{l+1},...,\lambda_{m}\right)  .$ Then the set
\begin{equation}
{\large \varphi}_{2}^{^{\prime\prime}}\left(  Z,\Lambda\right)  =\left\{
A\in\mathbb{H}_{a}^{n\times n}\left(  P,P\right)  \text{ }|\text{ }%
AZ=Z\Lambda\right\}  \label{AP}%
\end{equation}
is nonempty if and only if
\[
X_{1}\Phi X_{1}X_{1}^{+}=X_{1}\Phi,\text{ }X_{2}\Psi X_{2}X_{2}^{+}=X_{2}%
\Psi.
\]
In that case, the general expression of $A$ is
\[
A=\left[
\begin{array}
[c]{cc}%
U_{1}, & U_{2}%
\end{array}
\right]  \left[
\begin{array}
[c]{cc}%
0 & X_{1}\Phi X_{1}^{+}+W_{1}R_{X_{1}}\\
X_{2}\Psi X_{2}^{+}+W_{2}R_{X_{2}} & 0
\end{array}
\right]  \left[
\begin{array}
[c]{c}%
U_{1}^{\ast}\\
U_{2}^{\ast}%
\end{array}
\right]  ,
\]
where $W_{1},W_{2}$ are
arbitrary matrices over $\mathbb{H}$ with appropriate sizes.
\end{corollary}

\begin{remark}
We find it is not necessary that $P$ and $Q$ are hermitian involutionary
matrices in the process of dealing with inverse eigenproblems. So we can only
consider that $P$ and $Q$ are involutionary matrices and obtain some
corresponding conclusions. Furthermore, when $P=Q$ are only involutionary
matrix, we can get some results of inverse eigenproblems of \cite{Trench2}.
\end{remark}

\section{The approximation problem to the solution to the inverse
eigenproblems}

In this section, we consider the approximation problem: $\underset
{A\in{\large \varphi}}{\min\left\Vert A-E\right\Vert _{F}},$ where
$E\in\mathbb{H}^{m\times n}$ is a given matrix over $\mathbb{H}$\ and
${\large \varphi}$ is either $\mathbb{H}_{r}^{n\times n}\left(  P,Q\right)  $
or $\mathbb{H}_{a}^{n\times n}\left(  P,Q\right)  .$

\begin{lemma}
Suppose that $E\in\mathbb{H}^{m\times n},$ $\Gamma\in\mathbb{H}^{m\times m}$
and $\digamma\in\mathbb{H}^{n\times n}$ where $\Gamma^{2}=\Gamma=\Gamma^{\ast
}$ and $\digamma^{2}=\digamma=\digamma^{\ast}.$ Then the nearness problem
$\underset{X\in\mathbb{H}^{m\times n}}{\min}\left\Vert \Gamma X\digamma
-E\right\Vert _{F}$ is consistent if and only if
\[
\Gamma(X-E)\digamma=0.
\]
In that case,
\[
\underset{X\in\mathbb{H}^{m\times n}}{\min}\left\Vert \Gamma X\digamma
-E\right\Vert_{F}=\left\Vert \Gamma E\digamma-E\right\Vert _{F}%
\]

\end{lemma}

\begin{proof}
Note that
\[
\Gamma X\digamma-E=(\Gamma E\digamma-E)+\Gamma(X-E)\digamma.
\]
Then
\begin{equation}
(\Gamma X\digamma-E)^{\ast}\left(  \Gamma X\digamma-E\right)  =(\Gamma
E\digamma-E)^{\ast}(\Gamma E\digamma-E)+\left[  \Gamma(X-E)\digamma\right]
^{\ast}\Gamma(X-E)\digamma+G+G^{\ast} \label{w1}%
\end{equation}
where
\begin{align*}
G   =(\Gamma E\digamma-E)^{\ast}\Gamma(X-E)\digamma& =(\digamma E^{\ast}%
\Gamma-E^{\ast})\Gamma(X-E)\digamma\\
&  =\left(  \digamma-I\right)  E^{\ast}\Gamma(X-E)\digamma
\end{align*}
Suppose $\left(  \lambda,x\right)  $ is a right eigenpair for the matrix
$G\in\mathbb{H}^{n\times n},$ i.e., $Gx=x\lambda$ with $\lambda\neq0$. Since
$\digamma^{2}=\digamma,$ i.e., $\digamma(\digamma-I)=0,$ we obtain
\[
\digamma Gx=\digamma\left(  \digamma-I\right)  E^{\ast}\Gamma(X-E)\digamma
x=\digamma x\lambda=0
\]
i.e. $\digamma x=0.$ yielding $x=0.$ Hence, $G$ has nonzero right eigenvalues.
Consequently, taking traces in (\ref{w1}) yields
\[
\left\Vert \Gamma X\digamma-E\right\Vert _{F}^{2}=\left\Vert \Gamma
X\digamma-E\right\Vert _{F}^{2}+\left\Vert \Gamma(X-E)\digamma\right\Vert
_{F}^{2}.
\]
$\underset{X\in\mathbb{H}^{m\times n}}{\min}\left\Vert \Gamma X\digamma
-E\right\Vert _{F}$ is equivalent to $\underset{X\in\mathbb{H}^{m\times n}%
}{\min}\left\Vert \Gamma(X-E)\digamma\right\Vert _{F}.$

Clearly, when
$X=E+L_{\Gamma}YR_{\digamma},$\,\,\,\, $\underset{X\in\mathbb{H}^{m\times n}}%
{\min}\left\Vert \Gamma(X-E)\digamma\right\Vert _{F}=0$ where $Y$ is an
arbitrary matrix over $\mathbb{H}$ with appropriate size$.$ Hence
\[
\underset{X\in\mathbb{H}^{m\times n}}{\min}\left\Vert \Gamma X\digamma
-E\right\Vert_{F}=\left\Vert \Gamma E\digamma-E\right\Vert _{F}.
\]

\end{proof}

\begin{theorem}
Given a matrix $E\in\mathbb{H}^{n\times n}$;
\begin{equation}
\left[
\begin{array}
[c]{c}%
U_{1}^{\ast}\\
U_{2}^{\ast}%
\end{array}
\right]  E\left[
\begin{array}
[c]{cc}%
V_{1}, & V_{2}%
\end{array}
\right]  =\left[
\begin{array}
[c]{cc}%
E_{11} & E_{12}\\
E_{21} & E_{22}%
\end{array}
\right]  . \label{ac1}%
\end{equation}
where $U_{1},$ $U_{2},$ $V_{1},$ $V_{2}\ $are defined as (\ref{A1}) and
(\ref{A2}) and $E_{11}\in\mathbb{H}^{r_{1}\times r_{2}},$ $E_{12}\in
\mathbb{H}^{r_{1}\times\left(  n-r_{2}\right)  },$ $E_{21}\in\mathbb{H}%
^{\left(  n-r_{1}\right)  \times r_{2}},$ $E_{22}\in\mathbb{H}^{\left(
n-r_{1}\right)  \times\left(  n-r_{2}\right)  }$. Then the approximation
problem: $\underset{A\in{\large \varphi}_{1}}{\min}\left\Vert A-E\right\Vert
_{F}$ has a unique solution
\begin{equation}
A_{r}=\left[
\begin{array}
[c]{cc}%
U_{1}, & U_{2}%
\end{array}
\right]  \left[
\begin{array}
[c]{cc}%
Y_{1}\Phi X_{1}^{+}+E_{11}R_{X_{1}} & 0\\
0 & Y_{2}\Psi X_{2}^{+}+E_{22}R_{X_{2}}%
\end{array}
\right]  \left[
\begin{array}
[c]{c}%
V_{1}^{\ast}\\
V_{2}^{\ast}%
\end{array}
\right]  \label{c2}%
\end{equation}
where ${\large \varphi}_{1}$ are defined as (\ref{ak1}).
\end{theorem}

\begin{proof}
$A\in{\large \varphi}_{1}$, by Theorem 2.6 and (\ref{ac1}),
\begin{align*}
\left\Vert A-E\right\Vert _{F}^{2}  &
=\left\Vert \left[
\begin{array}
[c]{cc}%
Y_{1}\Phi X_{1}^{+}+W_{1}R_{X_{1}} & 0\\
0 & Y_{2}\Psi X_{2}^{+}+W_{2}R_{X_{2}}%
\end{array}
\right]  -\left[
\begin{array}
[c]{c}%
U_{1}^{\ast}\\
U_{2}^{\ast}%
\end{array}
\right]  E\left[
\begin{array}
[c]{cc}%
V_{1}, & V_{2}%
\end{array}
\right]  \right\Vert _{F}^{2}\\
&  =\left\Vert \left[
\begin{array}
[c]{cc}%
Y_{1}\Phi X_{1}^{+}+W_{1}R_{X_{1}} & 0\\
0 & Y_{2}\Psi X_{2}^{+}+W_{2}R_{X_{2}}%
\end{array}
\right]  -\left[
\begin{array}
[c]{cc}%
E_{11} & E_{12}\\
E_{21} & E_{22}%
\end{array}
\right]  \right\Vert _{F}^{2}\\
&  =\left\Vert \left[
\begin{array}
[c]{cc}%
Y_{1}\Phi X_{1}^{+}+W_{1}R_{X_{1}}-E_{11} & -E_{12}\\
-E_{21} & Y_{2}\Psi X_{2}^{+}+W_{2}R_{X_{2}}-E_{22}%
\end{array}
\right]  \right\Vert _{F}^{2}\\
&  =\left\Vert W_{1}R_{X_{1}}-\left(  E_{11}-Y_{1}\Phi X_{1}^{+}\right)
\right\Vert _{F}^{2}+\left\Vert W_{2}R_{X_{2}}-\left(  E_{22}-Y_{2}\Psi
X_{2}^{+}\right)  \right\Vert _{F}^{2}+\left\Vert E_{12}\right\Vert _{F}%
^{2}+\left\Vert E_{21}\right\Vert _{F}^{2}.
\end{align*}
Therefore, $\underset{A\in{\large \varphi}_{1}}{\min}\left\Vert A-E\right\Vert
_{F}^{2}$ is consistent is equivalent to
\begin{align}
&  \min\left\Vert W_{1}R_{X_{1}}-\left(  E_{11}-Y_{1}\Phi X_{1}^{+}\right)
\right\Vert _{F}^{2},\text{ }\label{RT1}\\
&  \min\left\Vert W_{2}R_{X_{2}}-\left(  E_{22}-Y_{2}\Psi X_{2}^{+}\right)
\right\Vert _{F}^{2} \label{RT2}%
\end{align}
are consistent. Since $R_{X_{_{i}}}^{2}=R_{X_{_{i}}}=R_{X_{_{i}}}^{\ast},$
$i=1,2,$ it foillows from Lemma 3.1 that (\ref{RT1}) is consistent if and only
if
\begin{equation}
\left[  W_{1}-\left(  E_{11}-Y_{1}\Phi X_{1}^{+}\right)  \right]  R_{X_{1}}=0
\label{AG1}%
\end{equation}
and
\begin{align*}
\min\left\Vert W_{1}R_{X_{1}}-\left(  E_{11}-Y_{1}\Phi X_{1}^{+}\right)
\right\Vert _{F}  &  =\left\Vert \left(  E_{11}-Y_{1}\Phi X_{1}^{+}\right)
-\left(  E_{11}-Y_{1}\Phi X_{1}^{+}\right)  R_{X_{1}}\right\Vert _{F}\\
&  =\left\Vert E_{11}X_{1}X_{1}^{+}-Y_{1}\Phi X_{1}^{+}\right\Vert _{F}%
\end{align*}
and (\ref{RT2}) is consistent if and only if
\begin{equation}
\left[  W_{2}-\left(  E_{22}-Y_{2}\Psi X_{2}^{+}\right)  \right]  R_{X_{2}}=0
\label{AG2}%
\end{equation}
and
\begin{align*}
\min\left\Vert W_{2}R_{X_{2}}-\left(  E_{22}-Y_{2}\Psi X_{2}^{+}\right)
\right\Vert _{F}  &  =\left\Vert \left(  E_{22}-Y_{2}\Psi X_{2}^{+}\right)
-\left(  E_{22}-Y_{2}\Psi X_{2}^{+}\right)  R_{X_{2}}\right\Vert _{F}\\
&  =\left\Vert E_{22}X_{2}X_{2}^{+}-Y_{2}\Psi X_{2}^{+}\right\Vert _{F}.
\end{align*}
It is easy to know the solutions of (\ref{AG1}) and (\ref{AG2}) are
\begin{align}
W_{1}  &  =E_{11}-Y_{1}\Phi X_{1}^{+}+T_{1}X_{1}X_{1}^{+},\label{AG3}\\
W_{2}  &  =E_{22}-Y_{2}\Psi X_{2}^{+}+T_{2}X_{2}X_{2}^{+}, \label{AG4}%
\end{align}
respectively, where $T_{1},$ $T_{2}$ are arbitary matrices with appropriate
sizes. Substituting (\ref{AG3}) and (\ref{AG4}) into (\ref{AK3}), we easily
get (\ref{c2}).
\end{proof}

\begin{corollary}
Given a matrix $E\in\mathbb{H}^{n\times n}$;
\[
\left[
\begin{array}
[c]{c}%
U_{1}^{\ast}\\
U_{2}^{\ast}%
\end{array}
\right]  E\left[
\begin{array}
[c]{cc}%
U_{1}, & U_{2}%
\end{array}
\right]  =\left[
\begin{array}
[c]{cc}%
E_{11} & E_{12}\\
E_{21} & E_{22}%
\end{array}
\right]  .
\]
where $U_{1},$ $U_{2}$ are defined as (\ref{A1}) and $E_{11}\in\mathbb{H}%
^{r_{1}\times r_{1}},$ $E_{12}\in\mathbb{H}^{r_{1}\times\left(  n-r_{1}%
\right)  },$ $E_{21}\in\mathbb{H}^{\left(  n-r_{1}\right)  \times r_{1}},$
$E_{22}\in\mathbb{H}^{\left(  n-r_{1}\right)  \times\left(  n-r_{1}\right)  }%
$. Then the approximation problem: $\underset{A\in{\large \varphi}%
_{1}^{^{\prime\prime}}}{\min}\left\Vert A-E\right\Vert _{F}$ has a unique
solution
\[
\tilde{A}_{r}=\left[
\begin{array}
[c]{cc}%
U_{1}, & U_{2}%
\end{array}
\right]  \left[
\begin{array}
[c]{cc}%
X_{1}\Phi X_{1}^{+}+E_{11}R_{X_{1}} & 0\\
0 & X_{2}\Psi X_{2}^{+}+E_{22}R_{X_{2}}%
\end{array}
\right]  \left[
\begin{array}
[c]{c}%
U_{1}^{\ast}\\
U_{2}^{\ast}%
\end{array}
\right]
\]
where ${\large \varphi}_{1}^{^{\prime\prime}}$ are defined as (\ref{ak4}).

\begin{theorem}
Given a matrix $E\in\mathbb{H}^{n\times n}$; Suppose 
\[
\left[
\begin{array}
[c]{c}%
U_{1}^{\ast}\\
U_{2}^{\ast}%
\end{array}
\right]  E\left[
\begin{array}
[c]{cc}%
V_{1}, & V_{2}%
\end{array}
\right]  =\left[
\begin{array}
[c]{cc}%
E_{11} & E_{12}\\
E_{21} & E_{22}%
\end{array}
\right]  ,
\]
where $U_{1},$ $U_{2},$ $V_{1},$ $V_{2}\ $are defined as (\ref{A1}) and
(\ref{A2}) and $E_{11}\in\mathbb{H}^{r_{1}\times r_{2}},$ $E_{12}\in
\mathbb{H}^{r_{1}\times\left(  n-r_{2}\right)  },$ $E_{21}\in\mathbb{H}%
^{\left(  n-r_{1}\right)  \times r_{2}},$ $E_{22}\in\mathbb{H}^{\left(
n-r_{1}\right)  \times\left(  n-r_{2}\right)  }$. Then the approximation
problem: $A_{a}=\arg\underset{A\in{\large \varphi}_{2}}{\min}\left\Vert A-E\right\Vert
_{F}$ has a unique solution
\[
A_{a}=\left[
\begin{array}
[c]{cc}%
U_{1}, & U_{2}%
\end{array}
\right]  \left[
\begin{array}
[c]{cc}%
0 & Y_{1}\Phi X_{1}^{+}+E_{12}R_{X_{1}}\\
Y_{2}\Psi X_{2}^{+}+E_{21}R_{X_{2}} & 0
\end{array}
\right]  \left[
\begin{array}
[c]{c}%
V_{1}^{\ast}\\
V_{2}^{\ast}%
\end{array}
\right]
\]
where ${\large \varphi}_{2}$ are defined as (\ref{AG}).
\end{theorem}
\end{corollary}

\begin{proof}
Combing Theorem 2.7, similarily to the proof of Theorem 3.2, we easily
complete the proof the Theorem.
\end{proof}

\begin{corollary}
Given a matrix $E\in\mathbb{H}^{n\times n}$; Suppose 
\[
\left[
\begin{array}
[c]{c}%
U_{1}^{\ast}\\
U_{2}^{\ast}%
\end{array}
\right]  E\left[
\begin{array}
[c]{cc}%
U_{1}, & U_{2}%
\end{array}
\right]  =\left[
\begin{array}
[c]{cc}%
E_{11} & E_{12}\\
E_{21} & E_{22}%
\end{array}
\right]  ,
\]
where $U_{1},$ $U_{2}\ $are defined as (\ref{A1}) and (\ref{A2}) and
$E_{11}\in\mathbb{H}^{r_{1}\times r_{1}},$ $E_{12}\in\mathbb{H}^{r_{1}%
\times\left(  n-r_{1}\right)  },$ $E_{21}\in\mathbb{H}^{\left(  n-r_{1}%
\right)  \times r_{1}},$ $E_{22}\in\mathbb{H}^{\left(  n-r_{1}\right)
\times\left(  n-r_{1}\right)  }$. Then the approximation problem:
$\tilde{A}_{a}=\arg\underset{A\in{\large \varphi}_{2}^{^{\prime\prime}}}{\min}\left\Vert
A-E\right\Vert _{F}$ has a unique solution
\[
\tilde{A}_{a}=\left[
\begin{array}
[c]{cc}%
U_{1}, & U_{2}%
\end{array}
\right]  \left[
\begin{array}
[c]{cc}%
0 & X_{1}\Phi X_{1}^{+}+E_{12}R_{X_{1}}\\
X_{2}\Psi X_{2}^{+}+E_{21}R_{X_{2}} & 0
\end{array}
\right]  \left[
\begin{array}
[c]{c}%
U_{1}^{\ast}\\
U_{2}^{\ast}%
\end{array}
\right]
\]
where ${\large \varphi}_{2}^{^{\prime\prime}}$ are defined as (\ref{AP}).
\end{corollary}

\bigskip

\section{\textbf{Conclusion}}

We in this paper used the matrix decomposition method to give the properties of the
generalized reflexive and antireflexive matrices with respect to a pair of
generalized reflection matrices over quaternion algebra. We studied the inverse eigenproblems and the approximation problems of the
generalized reflexive and antireflexive matrices with respect to a pair of
generalized reflection matrices. We also showed the explicit expressions and formulas.

\end{document}